\documentclass[preprint,1p,times,sort&compress]{elsarticle}
\usepackage{amsmath,amsfonts,mathrsfs}
\usepackage{amssymb}

\usepackage{amsthm}
\newtheorem{theorem}{Theorem}[section]
\newtheorem{lemma}{Lemma}[section]
\newtheorem{example}{Example}[section]
\newtheorem{Corollary}{Corollary}[section]
\newtheorem{rem}{Remark}[section]

\numberwithin{equation}{section}
\def\vector{\mathrm{vec}}

\journal{
        }

\begin{document}
\begin{frontmatter}
\title{Solutions and perturbation analysis of the matrix equation $X-\sum\limits_{i=1}^{m}A_{i}^*X^{-1}A_{i}=Q$
 \tnoteref{label1} }
\tnotetext[label1]{The work was supported in part by Natural Science Foundation of Shandong Province (ZR2012AQ004).}
\author
       {Jing Li
        }
\ead{xlijing@sdu.edu.cn}

\address
               {School of Mathematics and Statistics, Shandong University at Weihai, Weihai 264209, P.R. China}

\begin{abstract}
Consider the nonlinear matrix equation $X-\sum\limits_{i=1}^{m}A_{i}^*X^{-1}A_{i}=Q$. This paper shows that there exists a unique positive
definite solution to the equation without any restriction on $A_{i}$. Three perturbation bounds for the unique solution to the equation are evaluated. A
backward error of an approximate solution for the unique solution to the equation is derived.
Explicit expressions of the condition number for the unique solution to the equation are obtained. The theoretical results are illustrated by numerical examples.
\end{abstract}

\end{frontmatter}

\section{Introduction}
In this paper the nonlinear matrix equation
\begin{equation}\label{eq1}                                   
X-\sum\limits_{i=1}^{m}A_{i}^*X^{-1}A_{i}=Q
\end{equation}
is investigated, where $A_{1}, A_{2}, \ldots , A_{m}$ are $n\times n$ complex matrices, $m$ is a positive integer and $Q$ is a positive definite matrix. Here, $A_{i}^{*}$ denotes
the conjugate transpose of the matrix $A_{i}$.

This type of nonlinear matrix
equations arises in many practical applications. The equation $X-A^{*}X^{-1}A=Q$
which is representative of Eq.(\ref{eq1}) for $m=1$ comes from ladder networks, dynamic
programming, control theory, stochastic filtering, statistics and
so forth \cite{s1,s2,s3,s4,s5,s6}. When $m>1$, Eq.(\ref{eq1}) is recognized as playing an important role in
solving a system of linear equations in many physical calculations.

For the equation $X\pm A^{*}X^{-1}A=Q,$ there were many contributions
in the literature to the theory, applications and numerical
solutions \cite{s8,s9,s10,s11,s12,s13,s22,s23,s24,s25,s26}.
The general equations such as
 $X\pm A^*X^{-2}A=Q$ \cite{s14,s15,s16,s17}, $X^{s}\pm A^*X^{-t}A=Q$ \cite{s18,s19,s20,s21}
 and $X\pm A^*X^{-q}A=Q$ \cite{s27,s28,s39} were also investigated by many scholars. In addition,
He and Long \cite{s30} and Duan et al. \cite{s31} have studied the similar equation
$X+\sum\limits_{i=1}^{m}A_{i}^*X^{-1}A_{i}=I$. sarhan et al. \cite{s38} discussed the existence of extremal positive definite solution of the matrix equation $X^{r}+\sum\limits_{i=1}^{m}A_{i}^{*}X^{\delta_{i}}A_{i}=I$. Duan et al. \cite{s29} proved that the equation
$X-\sum\limits_{i=1}^{m}A_{i}^*X^{\delta_{i}}A_{i}=Q\;(0<|\delta_{i}|<1)$
has a unique positive definite solution. They also proposed an iterative method
for obtaining the unique positive definite solution. However, to our best knowledge, there has been no perturbation analysis for Eq.(\ref{eq1}) in the known literatures.

The rest of the paper is organized as follows. Section 2 gives some preliminary lemmas that will be needed to develop this work.
Section 3 proves the existence of a unique positive definite solution to Eq.(\ref{eq1}) without any restriction on $A_{i}$.
Section 4 gives three perturbation bounds for
 the unique solution to Eq.(\ref{eq1}).
 Section 5 derives a
backward error of an approximate solution for the unique solution to Eq.(\ref{eq1}).
 Furthermore, in Section 6,  the condition number of the unique solution to Eq.(\ref{eq1}) is discussed. Finally, several numerical
  examples are presented in Section 7.

We denote by $\mathcal{C}^{n\times n}$ the set of $n\times n$
complex matrices, by $\mathcal{H}^{n\times n}$ the set of $n\times n$
Hermitian matrices, by $I$ the identity matrix, by $\textbf{i}$ the imaginary unit, by $\|\cdot\|$ the spectral
norm, by $\|\cdot\|_{F}$ the Frobenius norm and by $\lambda_{\max}(M)$ and
$\lambda_{\min}(M)$ the maximal and minimal eigenvalues of $M$,
respectively. For $A=(a_{1},\dots, a_{n})=(a_{ij})\in
\mathcal{C}^{n\times n}$ and a matrix $B$, $A\otimes B=(a_{ij}B)$
is a Kronecker product, and $\vector A$
is a vector defined by $\vector A=(a_{1}^{T},\dots,
a_{n}^{T})^{T}$. For $X, Y\in\mathcal{H}^{n\times n}$, we write $X\geq Y$(resp. $X>Y)$ if
$X-Y$ is Hermitian positive semi-definite (resp. definite).
\section{Preliminaries}
\begin{lemma}\label{lem1}                   
\cite{s32}. If $A\geq B>0,$ then $0<A^{-1}\leq B^{-1}.$
\end{lemma}

\begin{lemma}\label{lem2}                   
\cite{s26}. For every positive definite matrix $X\in\mathcal{H}^{n\times n}$, if $X+\Delta X\geq (1/\nu) X>0,$ then
$$\|X^{-\frac{1}{2}}A^{*}((X+\Delta X)^{-1}-X^{-1})AX^{-\frac{1}{2}}\|
\leq (\|X^{-\frac{1}{2}}\Delta X X^{-\frac{1}{2}}\|+\nu \|X^{-\frac{1}{2}}\Delta X X^{-\frac{1}{2}}\|^{2})\|X^{-\frac{1}{2}}AX^{-\frac{1}{2}}\|^{2}.$$
\end{lemma}

\begin{lemma}\label{lem4}                   
\cite{s35}. The matrix differentiation has the following properties:
\begin{enumerate}
  \item[(1)]$\textmd{d}(F_{1}\pm F_{2})=\textmd{d}F_{1}\pm \textmd{d}F_{2};$
  \item[(2)]$d(kF)=k(dF),$ where $k$ is a complex number;
  \item[(3)]$d(F^{*})=(dF)^{*};$
  \item[(4)] $d(F_{1}F_{2}F_{3})=(dF_{1})F_{2}F_{3}+F_{1}(dF_{2})F_{3}+F_{1}F_{2}(dF_{3});$
  \item[(5)] $dF^{-1}=-F^{-1}(dF)F^{-1};$
  \item[(6)] $dF=0,$ where $F$ is a constant matrix.
\end{enumerate}
\end{lemma}
\section{Positive definite solution of the matrix Eq.(1.1)}
In this section, the existence of a unique positive definite solution of Eq.(\ref{eq1}) is proved. Moreover, some properties of the unique positive definite solution of Eq.(\ref{eq1}) are obtained.

\begin{theorem} \label{thm1}                                                           
If $F(X)=Q+\sum\limits_{i=1}^{m}A_{i}^*X^{-1}A_{i},$ then $F([Q, Q+\sum\limits_{i=1}^{m}A_{i}^*Q^{-1}A_{i}])\subseteq
 [Q, Q+\sum\limits_{i=1}^{m}A_{i}^*Q^{-1}A_{i}].$
\end{theorem}

\begin{proof}{Let $\Omega =[Q,Q+\sum\limits_{i=1}^{m}A_{i}^*Q^{-1}A_{i}]$. By Lemma \ref{lem1}, we obtain
$0<X^{-1}\leq Q^{-1}$ for every $X\in \Omega$. Applying Eq.(\ref{eq1}) yields
$Q\leq F(X)\leq Q+\sum\limits_{i=1}^{m}A_{i}^*Q^{-1}A_{i}$. Therefore $F(\Omega)\subseteq\Omega$.}
\end{proof}

\begin{theorem}     \label{thm5}                                                                      
There exists a unique positive definite solution $X$ to
Eq.(\ref{eq1}) and the iteration
\begin{equation}      \label{eq6}                                                            
X_0>0, \;\;X_n=Q+\sum\limits_{i=1}^{m}A_{i}^*X_{n-1}^{-1}A_{i},\;\;\;n=1,2,\cdots
\end{equation}
 converges to $X$.
\end{theorem}

To prove the above theorem, we first verify the following lemma.

\begin{lemma} \label{lem3}                                                       
 Let $F(X)=Q+\sum\limits_{i=1}^{m}A_{i}^*X^{-1}A_{i}$. If $0<t<1$ and $X\in[Q,Q+\sum\limits_{i=1}^{m}A_{i}^*Q^{-1}A_{i}],$ then
\begin{equation*}
F^2(tX)\geq t(1+\eta(t))F^2(X),
\end{equation*}
where$$\eta(t)=\frac{(1-t)\lambda_{\min}(Q)}{t\left(\lambda_{\max}(Q)+
 \frac{\sum\limits_{i=1}^{m}\lambda_{\max}(A_{i}^{*}A_{i})}{\lambda_{\min}(Q)}\right)}.$$
\end{lemma}

\begin{proof} {According to Theorem \ref{thm1}, for every $ X\in[Q,Q+\sum\limits_{i=1}^{m}A_{i}^*Q^{-1}A_{i}]$, we have $F(X)\in[Q,Q+\sum\limits_{i=1}^{m}A_{i}^*Q^{-1}A_{i}]$ and $F^{2}(X)\in[Q,Q+\sum\limits_{i=1}^{m}A_{i}^*Q^{-1}A_{i}]$. Hence we have
\begin{eqnarray*}
 & & F^{2}(t X)-t(1+\eta(t))F^{2}(X)\\
 &=&(1-t)Q+t\sum\limits_{i=1}^{m}A_{i}^*\left[(tQ+\sum\limits_{i=1}^{m}A_{i}^*X^{-1}A_{i})^{-1}
 -(Q+\sum\limits_{i=1}^{m}A_{i}^*X^{-1}A_{i})^{-1}\right]A_{i}\\
 &&-\frac{(1-t)\lambda_{\min}(Q)}{\lambda_{\max}(Q)+
 \frac{\sum\limits_{i=1}^{m}\lambda_{\max}(A_{i}^{*}A_{i})}{\lambda_{\min}(Q)}}F^{2}(X)\\
 &\geq&
 (1-t)\lambda_{\min}(Q)I-\frac{(1-t)\lambda_{\min}(Q)}{\lambda_{\max}(Q)+
 \frac{\sum\limits_{i=1}^{m}\lambda_{\max}(A_{i}^{*}A_{i})}{\lambda_{\min}(Q)}}
 \left(\lambda_{\max}(Q)+
 \frac{\sum\limits_{i=1}^{m}\lambda_{\max}(A_{i}^{*}A_{i})}{\lambda_{\min}(Q)}\right)I=0, \;0 < t < 1.
\end{eqnarray*}}
\end{proof}

{\bf Proof of Theorem \ref{thm5}}\; Let $F(X)=Q+\sum\limits_{i=1}^{m}A_{i}^*X^{-1}A_{i}$ and
$\Omega = [Q, I+\sum\limits_{i=1}^{m}A_{i}^*Q^{-1}A_{i}]$. The proof will be divided into
two steps. \

(1) We prove the special case of Theorem \ref{thm5} when
$X_{0}=Q.$

It is easy to check that
$$Q\leq X_1 = Q + \sum\limits_{i=1}^{m}A_{i}^*X_{0}^{-1}A_{i}=F(Q)=Q + \sum\limits_{i=1}^{m}A_{i}^*Q^{-1}A_{i},$$
$$Q\leq X_2 = Q+ \sum\limits_{i=1}^{m}A_{i}^*X_{1}^{-1}A_{i}=F^2(Q)\leq F(Q),$$
$$F^2(Q)\leq X_3=Q + \sum\limits_{i=1}^{m}A_{i}^*X_{2}^{-1}A_{i}=F^3(Q)\leq F(Q),$$
$$ F^2(Q)\leq X_4=Q + \sum\limits_{i=1}^{m}A_{i}^*X_{3}^{-1}A_{i}=F^4(Q)\leq F^3(Q).$$
By induction, it yields that
$$ Q\leq F^{2k}(Q)\leq F^{2k+2}(Q)\leq F^{2k+1}(Q)\leq
F^{2k-1}(Q)\leq Q + \sum\limits_{i=1}^{m}A_{i}^*Q^{-1}A_{i}, k\in \mathbb{Z}^{+}.$$ Hence the
sequences $\{F^{2k}(Q)\}$ and $\{F^{2k+1}(Q)\}$ are convergent.
Let $\lim\limits_{k\to\infty}F^{2k}(Q)=X^{(1)},$
$\lim\limits_{k\to\infty}F^{2k+1}(Q)=X^{(2)}.$ It is clear that
$X^{(1)}$
and $X^{(2)}$ are positive fixed points of $F^2(X)$. \\

In the following part, we first prove that $X^{(1)}=X^{(2)}$. Suppose that $Y_1$ and $Y_2$ are
two positive fixed points of $F^2$ in $\Omega$. We compute
\begin{eqnarray*}
Y_1&=&F^{2}(Y_1) \geq Q \geq
\displaystyle\frac{1}{1+\displaystyle\frac{\sum\limits_{i=1}^{m}\lambda_{\max}(A_{i}^{*}A_{i})}{\lambda_{\min}^{2}(Q)}}\,(Q + \sum\limits_{i=1}^{m}A_{i}^*Q^{-1}A_{i})\\
&\geq &
\displaystyle\frac{1}{1+\displaystyle\frac{\sum\limits_{i=1}^{m}\lambda_{\max}(A_{i}^{*}A_{i})}{\lambda_{\min}^{2}(Q)}}\,F^{2}(Y_2)
=t\,Y_2,\;\;t=\displaystyle\frac{1}{1+\displaystyle\frac{\sum\limits_{i=1}^{m}\lambda_{\max}(A_{i}^{*}A_{i})}{\lambda_{\min}^{2}(Q)}}.
\end{eqnarray*}
Let $t_{0} = \sup \{t| Y_1 \geq t Y_2 \}$. Then $1\leq  t_{0} <
+\infty$. On the contrary, suppose that $0 < t_{0} < 1$. Then $Y_1
\geq t_{0}Y_2$. According to Lemma \ref{lem3} and the monotonicity of
$F^2(X)$, we have $$Y_1 = F^{2}(Y_1) \geq F^{2}(t_{0}Y_2) \geq
(1+\eta(t_{0}))t_{0}F^{2}(Y_2)=(1+\eta(t_{0}))t_{0}Y_{2}.$$ By the definition of $\eta(t),$ we
obtain $(1+\eta(t_{0}))\,t_{0}
> t_{0}$, which is a contradiction to the definition of $t_{0}$. Hence we
have $t_{0} \geq 1$ and  $Y_1 \geq Y_2$. Similarly, we get $Y_1 \leq
Y_2$. Therefore $Y_1 = Y_2$, i.e., the equation $X = F^2(X)$ has
only one positive definite solution. Hence
$X^{(1)}=X^{(2)}$.

Second, we prove that $\lim\limits_{n\to\infty}X_{n}$ is the
unique fixed point of F in $\Omega$. By $X^{(1)}=X^{(2)}$, it
follows that $X^{(1)}=X^{(2)}=\lim\limits_{n\to\infty}F^n(Q)$ is the unique fixed point of
$F^{2}$. Moreover, the positive definite solution of equation
$F(X) = X$ solves $X = F^{2}(X)$. Therefore $F(X) = X$ has only
one positive definite solution and
$\lim\limits_{n\to\infty}F^n(Q)=\lim\limits_{n\to\infty}X_{n}$ is
the unique fixed point of $F$.

(2) We prove the case of Theorem \ref{thm5} when $X_{0}>0.$

From
iteration (\ref{eq6}), we obtain

    $$X_1\geq Q ,$$
    $$Q\leq X_2=Q+ \sum\limits_{i=1}^{m}A_{i}^*X_1^{-1}A_{i}\leq Q+\sum\limits_{i=1}^{m}A_{i}^*Q^{-1}A_{i}=F(Q)$$
and
$$F^{2}(Q)\leq X_{3}=I+\sum\limits_{i=1}^{m}A_{i}^{*}X_{2}^{-1}A_{i}\leq F(Q).$$
By induction, we have
\begin{equation*}
    F^{2k}(Q)\leq X_{2k+1}\leq F^{2k-1}(Q)\;\;\; \text{and}\;\;\;F^{2k-2}(Q)\leq X_{2k}\leq F^{2k-1}(Q).
\end{equation*}
Therefore
\begin{equation*}
    \lim_{k\to\infty}X_{k}=\lim_{n\to\infty}F^{n}(Q).
\end{equation*}
It follows that $\lim\limits_{k\to\infty}X_{k}$ is the unique
positive definite solution of Eq.(\ref{eq1}).
$\;\;\square$

\begin{theorem} \label{thm2}                                     
If $X$ is a positive definite solution of Eq.(\ref{eq1}), then $Q\leq X\leq Q+\sum\limits_{i=1}^{m}A_{i}^*Q^{-1}A_{i}.$
\end{theorem}

\begin{proof}{That $X$ is a positive definite solution of Eq.(\ref{eq1}) implies $X>0.$ Then $X^{-1}>0$ and
$A_{i}^{*}X^{-1}A_{i}\geq 0.$ Hence $X=Q+\sum\limits_{i=1}^{m}A_{i}^*X^{-1}A_{i}\geq Q.$
 Consequently, $X^{-1}\leq Q^{-1}$ and $X\leq Q+\sum\limits_{i=1}^{m}A_{i}^*Q^{-1}A_{i}$.}
\end{proof}
\begin{theorem} \label{thm3}                                        
Every positive definite solution $X$ of Eq.(\ref{eq1}) is in $[
\beta I, \alpha I ]$, where $\alpha$ and $\beta$ are respectively the solutions of
the following equations
\begin{equation}\label{eq2}                                              
x=\lambda_{\max}(Q)+\frac{\sum\limits_{i=1}^{m}\lambda_{\max}(A_{i}^{*}A_{i})}
{\lambda_{\min}(Q)+\displaystyle\frac{\sum\limits_{i=1}^{m}\lambda_{\min}(A_{i}^{*}A_{i})}{x}},
\end{equation}
\begin{equation}\label{eq33}
x=\lambda_{\min}(Q)+\frac{\sum\limits_{i=1}^{m}\lambda_{\min}(A_{i}^{*}A_{i})}
{\lambda_{\max}(Q)+\displaystyle\frac{\sum\limits_{i=1}^{m}\lambda_{\max}(A_{i}^{*}A_{i})}{x}}.
\end{equation}
Moreover,
\begin{equation}\label{eq10}                                                
\lambda_{\min}(Q)\leq\beta\leq\alpha.
\end{equation}
\end{theorem}

\begin{proof}
{We define the sequences $\{\alpha_{n}\}$ and $\{\beta_{n}\}$ as follows:
\begin{equation}\label{eq3}                                                 
\beta_{0}=\lambda_{\min}(Q),\;\;\alpha_{n}=\lambda_{\max}(Q)+\frac{\sum\limits_{i=1}^{m}\lambda_{\max }(A_{i}^{*}A_{i})}{\beta_{n}},
\;\;\beta_{n+1}=\lambda_{\min}(Q)+\frac{\sum\limits_{i=1}^{m}\lambda_{\min }(A_{i}^{*}A_{i})}{\alpha_{n}},\;n=0,1,2,\cdots.
\end{equation}
From (\ref{eq3}), it follows that
\begin{eqnarray*}
\beta_{0}\leq\lambda_{\max}(Q)\leq \alpha_{0} &=& \lambda_{\max}(Q)+\frac{\sum\limits_{i=1}^{m}\lambda_{\max }(A_{i}^{*}A_{i})}{\lambda_{\min}(Q)},\\
 \lambda_{\min}(Q)=\beta_{0}\leq\beta_{1}&=&  \lambda_{\min}(Q)+\frac{\sum\limits_{i=1}^{m}\lambda_{\min}(A_{i}^{*}A_{i})}{\alpha_{0}}\leq \lambda_{\min}(Q)+\frac{\sum\limits_{i=1}^{m}\lambda_{\min }(A_{i}^{*}A_{i})}{\lambda_{\max}(Q)},\\
\beta_{0} \leq\lambda_{\max}(Q)\leq\alpha_{1} &=& \lambda_{\max}(Q)+\frac{\sum\limits_{i=1}^{m}\lambda_{\max}(A_{i}^{*}A_{i})}{\beta_{1}}
  \leq \lambda_{\max}(Q)+\frac{\sum\limits_{i=1}^{m}\lambda_{\max }(A_{i}^{*}A_{i})}{\lambda_{\min}(Q)}=\alpha_{0}.
  \end{eqnarray*}
  We suppose that
  $\lambda_{\max}(Q)\leq\alpha_{k}\leq\alpha_{k-1}$ and $\lambda_{\min}(Q)\leq\beta_{k-1}\leq\beta_{k}\leq \lambda_{\min}(Q)+\displaystyle\frac{\sum\limits_{i=1}^{m}\lambda_{\min }(A_{i}^{*}A_{i})}{\lambda_{\max}(Q)}.$ Then
  \begin{eqnarray*}
 \lambda_{\min}(Q)&\leq&\beta_{k}=\lambda_{\min}(Q)+
 \frac{\sum\limits_{i=1}^{m}\lambda_{\min}(A_{i}^{*}A_{i})}{\alpha_{k-1}}\\
  &\leq& \beta_{k+1}=\lambda_{\min}(Q)+\frac{\sum\limits_{i=1}^{m}\lambda_{\min}(A_{i}^{*}A_{i})}{\alpha_{k}} \leq \lambda_{\min}(Q)+\frac{\sum\limits_{i=1}^{m}\lambda_{\min }(A_{i}^{*}A_{i})}{\lambda_{\max}(Q)},\\
  \lambda_{\max}(Q)&\leq&\alpha_{k+1}= \lambda_{\max}(Q)+\frac{\sum\limits_{i=1}^{m}\lambda_{\max}(A_{i}^{*}A_{i})}{\beta_{k+1}}\\
   &\leq&\alpha_{k} =\lambda_{\max}(Q)+\frac{\sum\limits_{i=1}^{m}\lambda_{\max}(A_{i}^{*}A_{i})}{\beta_{k}}.
   \end{eqnarray*}
Hence, for each $k$ we have $\lambda_{\max}(Q)\leq \alpha_{k+1}\leq\alpha_{k}$ and $\lambda_{\min}(Q)\leq\beta_{k}\leq\beta_{k+1}\leq \lambda_{\min}(Q)+\frac{\sum\limits_{i=1}^{m}\lambda_{\min }(A_{i}^{*}A_{i})}{\lambda_{\max}(Q)},$ which imply that the sequences $\{\alpha_{n}\}$ and $\{\beta_{n}\}$ are monotonic and bounded. Therefore, they are convergent to
positive numbers.
Let $$\alpha=\lim_{n\rightarrow\infty}\alpha_{n},\;\;\beta=\lim_{n\rightarrow\infty}\beta_{n}.$$
Taking limits in (\ref{eq3}) yields

\begin{equation}\label{eq11}                                                            
  \alpha=\lambda_{\max}(Q)+\frac{\sum\limits_{i=1}^{m}\lambda_{\max}(A_{i}^{*}A_{i})}{\beta},\;\;
\beta=\lambda_{\min}(Q)+\frac{\sum\limits_{i=1}^{m}\lambda_{\min}(A_{i}^{*}A_{i})}{\alpha},
\end{equation}
which imply
$$\alpha=\lambda_{\max}(Q)+\frac{\sum\limits_{i=1}^{m}\lambda_{\max}(A_{i}^{*}A_{i})}
{\lambda_{\min}(Q)+\frac{\sum\limits_{i=1}^{m}\lambda_{\min}(A_{i}^{*}A_{i})}{\alpha}},\;\;\;
\beta=\lambda_{\min}(Q)+\frac{\sum\limits_{i=1}^{m}\lambda_{\min}(A_{i}^{*}A_{i})}
{\lambda_{\max}(Q)+\frac{\sum\limits_{i=1}^{m}\lambda_{\max}(A_{i}^{*}A_{i})}{\beta}}.$$
Therefore $\alpha$ and $\beta$ satisfy (\ref{eq2}) and (\ref{eq33}), respectively.
We will prove that $X\in [\beta I,\;\alpha I]$ for any positive definite solution $X.$
According to Theorem \ref{thm2} and the sequences in (\ref{eq3}), we have
$$\beta_{0}I\leq Q\leq X\leq(\lambda_{\max}(Q)+\displaystyle\frac{\sum\limits_{i=1}^{m}\lambda_{\max}(A_{i}^{*}A_{i})}
{\lambda_{\min}(Q)})I=\alpha_{0}I$$ for each  positive definite solution $X.$
From $X=Q+\sum\limits_{i=1}^{m}A_{i}^{*}X^{-1}A_{i},$ it follows that
$X=Q+\sum\limits_{i=1}^{m}A_{i}^{*}(Q+\sum\limits_{i=1}^{m}A_{i}^{*}X^{-1}A_{i})^{-1}A_{i}.$
Hence
\begin{equation}\label{eq4}                                                                   
\left(\lambda_{\min}(Q)+\frac{\sum\limits_{i=1}^{m}\lambda_{\min}(A_{i}^{*}A_{i})}
{\lambda_{\max}(Q)+\frac{\sum\limits_{i=1}^{m}\lambda_{\max}(A_{i}^{*}A_{i})}{\lambda_{\min}(X)}}\right)I\leq X\leq
\left(\lambda_{\max}(Q)+\frac{\sum\limits_{i=1}^{m}\lambda_{\max}(A_{i}^{*}A_{i})}
{\lambda_{\min}(Q)+\frac{\sum\limits_{i=1}^{m}\lambda_{\min}(A_{i}^{*}A_{i})}{\lambda_{\max}(X)}}\right)I.
\end{equation}
Using $\beta_{0}I\leq X\leq\alpha_{0}I,$ we obtain $\beta_{0}\leq\lambda_{\min}(X)$ and $\lambda_{\max}(X)\leq\alpha_{0}.$ Applying the inequality in (\ref{eq4}) yields $\beta_{1}I\leq X\leq\alpha_{1}I.$
By induction, it yields that $\beta_{n}I\leq X\leq\alpha_{n}I.$ Taking limits on both sides of the above inequality, we have $\beta I\leq X\leq\alpha I.$
}\end{proof}

\begin{Corollary}    \label{cor1}                                              
Every positive definite solution of Eq.(\ref{eq1}) is in $\left[Q+\displaystyle\frac{1}{\alpha}\sum\limits_{i=1}^{m}A_{i}^{*}A_{i},\;\;
Q+\displaystyle\frac{1}{\beta}\sum\limits_{i=1}^{m}A_{i}^{*}A_{i}\right],$ where $\alpha$ and $\beta$ are defined as in Theorem \ref{thm3}.
\end{Corollary}

\begin{proof}{We suppose that $X$ is a positive definite solution of Eq.(\ref{eq1}). By Theorem \ref{thm3}, it follows that
 \begin{equation}\label{eq5}                                                     
 \lambda_{\min}(Q)\leq\beta\leq \lambda_{\min}(X),\;\;\lambda_{\max}(Q)\leq\lambda_{\max}(X)\leq\alpha.
\end{equation}
Using $X=Q+\sum\limits_{i=1}^{m}A_{i}^{*}X^{-1}A_{i}$, we obtain $Q+\displaystyle\frac{\sum\limits_{i=1}^{m}A_{i}^{*}A_{i}}{\lambda_{\max}(X)}\leq X\leq
Q+\displaystyle\frac{\sum\limits_{i=1}^{m}A_{i}^{*}A_{i}}{\lambda_{\min}(X)}$. Applying inequality (\ref{eq5})
yields $Q+\displaystyle\frac{1}{\alpha}\sum\limits_{i=1}^{m}A_{i}^{*}A_{i}\leq X\leq
Q+\displaystyle\frac{1}{\beta}\sum\limits_{i=1}^{m}A_{i}^{*}A_{i}.$}
\end{proof}

\begin{rem}                                                           

Applying (\ref{eq11}), we obtain
\begin{eqnarray*}
&&Q+\displaystyle\frac{1}{\beta}\sum\limits_{i=1}^{m}A_{i}^{*}A_{i}\leq \left(\lambda_{\max}(Q)+\frac{\sum\limits_{i=1}^{m}\lambda_{\max}(A_{i}^{*}A_{i})}{\beta}\right)I=\alpha I,\\
&&Q+\displaystyle\frac{1}{\alpha}\sum\limits_{i=1}^{m}A_{i}^{*}A_{i}\geq
\left(\lambda_{\min}(Q)+\frac{\sum\limits_{i=1}^{m}\lambda_{\min}(A_{i}^{*}A_{i})}{\alpha}\right)I=\beta I.
\end{eqnarray*}
That is to say, the estimate of positive definite solution in Corollary \ref{cor1} is more precise than that in Theorem \ref{thm3}.
\end{rem}

\section{Perturbation bounds }

Here we consider the perturbed equation
\begin{equation}\label{eq7}                                                                
\widetilde{X}-\sum\limits_{i=1}^{m}\widetilde{A_{i}}^*\widetilde{X}^{-1}\widetilde{A_{i}}=\widetilde{Q},
\end{equation}
where $\widetilde{A_{i}}$, $\widetilde{Q}$ are  small perturbations
of $A_{i}$ and $Q$ in Eq.(\ref{eq1}), respectively. We assume that $X$
and $\widetilde{X}$ are the solutions of Eq.(\ref{eq1}) and
Eq.(\ref{eq7}), respectively. Let $\Delta X=\widetilde{X}-X$, $\Delta Q=\widetilde{Q}-Q$ and $\Delta A_{i}=\widetilde{A_{i}}-A_{i}$.

In this section we develop three perturbation bounds for the
solution of Eq.(\ref{eq1}). To begin with, a relative perturbation bound for the unique solution $X$ of
Eq.(\ref{eq1}) is derived  . The perturbation bound in
Theorem \ref{thm6} does not need any knowledge of the actual solution $X$
of Eq.(\ref{eq1}). Secondly, based on the matrix differentiation, we use the techniques developed in \cite{s31}
to derive another perturbation bound in Theorem \ref{thm8}. Finally, based on the operator theory, we obtain a sharper perturbation bound in
Theorem \ref{thm9}.

The next theorem generalizes Theorem 3.2 in Li and Zhang \cite{s26}
with $m=1$ to arbitrary integer $m\geq 1$.

\begin{theorem}\label{thm6}                           
 Let $b=\beta^{2}+\beta \;\|\Delta Q\|-\sum\limits_{i=1}^{m}\|A_{i}\|^{2},
s=\sum\limits_{i=1}^{m}\|\Delta A_{i}\|\;(2\|A_{i}\|+\|\Delta A_{i}\|)$. If
\begin{equation}\label{eq8}                              
0<b<2\beta^{2} \;\;{and} \;\;b^{2}-4\beta^{2}\;(\beta\;\|\Delta Q\|+s)\geq 0,
 \end{equation}
then
\begin{equation}\label{eq9}                              
\frac{\|\widetilde{X}-X\|}{\|X\|}\,\leq
\varrho\,\sum\limits_{i=1}^{m}\|\Delta A_{i}\|+\omega \|\Delta Q\|\equiv \xi_{1},
\end{equation}
where
$$\varrho=\displaystyle\frac{2s}{\sum\limits_{i=1}^{m}\|\Delta A_{i}\|(b+\sqrt{b^{2}-4\beta^{2}\;(\beta\;\|\Delta Q\|+s)})},\;\;\;\;\omega=\frac{2\beta}{b+\sqrt{b^{2}-4\beta^{2}\;(\beta\;\|\Delta Q\|+s)}}.$$

\end{theorem}

\begin{proof}  Let
$$\Omega=\{\Delta X \in\mathcal{H}^{n\times n}:\,\,\|
X^{-1/2}\Delta X X^{-1/2} \|\leq \varrho\,\sum\limits_{i=1}^{m}\|\Delta
A_{i}\|+\omega\|\Delta Q\|\;\}.$$ Obviously, $\Omega$ is a nonempty
bounded convex closed set. Let $$f(\Delta
X)=\sum\limits_{i=1}^{m}(\widetilde{A_{i}}^*(X+\Delta
X)^{-1}\widetilde{A_{i}}-A_{i}^*X^{-1}A_{i})+\Delta Q,\;\;\Delta
X\in\Omega.$$ Evidently, $f: \Omega\mapsto\mathcal{H}^{n\times n}$
is continuous. We will prove that $f(\Omega)\subseteq\Omega$.

For every $\Delta X\in\Omega,$ that is  $\|X^{-1/2}\Delta XX^{-1/2}\|\leq
\varrho\,\sum\limits_{i=1}^{m}\|\Delta A_{i}\|+\omega \|\Delta Q\|.$ Thus
$$X^{-1/2}\Delta XX^{-1/2}\geq (-\varrho \,\sum\limits_{i=1}^{m}\|\Delta A_{i}\|-\omega\|\Delta Q\| )I,$$
$$X+\Delta X\geq (1-\varrho \sum\limits_{i=1}^{m}\|\Delta A_{i}\|-\omega\Delta Q\|)X.$$
According to (\ref{eq8}) and (\ref{eq9}), we have $$\varrho\sum\limits_{i=1}^{m}\|\Delta A_{i}\|+\omega\|\Delta Q\|=\frac{2(\beta\;\|\Delta Q\|+s)}{b+\sqrt{b^{2}-4\beta^{2}\;(\beta\;\|\Delta Q\|+s)}}\leq \frac{2(\beta\;\|\Delta Q\|+s)}{b}\leq\frac{b}{2\beta^{2}}<1.$$ Therefore $$(1-\varrho\sum\limits_{i=1}^{m}\|\Delta A_{i}\|-\omega\|\Delta Q\|)X>0.$$
From Lemma \ref{lem2} and $X\geq\beta I$, it
follows that
\begin{eqnarray*}\label{eq16}                                                         
&&\left \| X^{-\frac{1}{2}}\left[\sum\limits_{i=1}^{m}A_{i}^*\left((X+\Delta
X\right)^{-1}-X^{-1})A_{i}\right]X^{-\frac{1}{2}}\right\|\nonumber \\
&\leq &\left(\| X^{-\frac{1}{2}}\Delta X X^{-\frac{1}{2}}\| +
\frac{\|X^{-\frac{1}{2}}\Delta X
X^{-\frac{1}{2}}\|^{2}}{1-\varrho\;\sum\limits_{i=1}^{m}\|\Delta
A_{i}\|-\omega\|\Delta Q\|}\right)\left(\sum\limits_{i=1}^{m}\|X^{-\frac{1}{2}}A_{i}X^{-\frac{1}{2}}\|^{2}\right)\\
&\leq &\left(\| X^{-\frac{1}{2}}\Delta X X^{-\frac{1}{2}}\| +
\frac{\|X^{-\frac{1}{2}}\Delta X
X^{-\frac{1}{2}}\|^{2}}{1-\varrho\;\sum\limits_{i=1}^{m}\|\Delta
A_{i}\|-\omega\|\Delta Q\|}\right)\left(\frac{1}{\beta^{2}}\sum\limits_{i=1}^{m}\|A_{i}\|^{2}\right).\nonumber
\end{eqnarray*}
Therefore
\begin{eqnarray*}
&&\left\| X^{-\frac{1}{2}}f(\Delta
X)X^{-\frac{1}{2}}\right\|\\
&=&\left\| X^{-\frac{1}{2}}\left[\sum\limits_{i=1}^{m}\widetilde{A_{i}}^*\left((X+\Delta
X)^{-1}-X^{-1}\right)\widetilde{A_{i}}\right]X^{-\frac{1}{2}}+X^{-\frac{1}{2}}\Delta QX^{-\frac{1}{2}}\right\|\\
&\leq &\left\| \sum\limits_{i=1}^{m}X^{-\frac{1}{2}}A_{i}^*((X+\Delta
X)^{-1}-X^{-1})A_{i}X^{-\frac{1}{2}}\right\|+\|X^{-\frac{1}{2}}\Delta QX^{-\frac{1}{2}}\|\\
&&+\left\|\sum\limits_{i=1}^{m}X^{-\frac{1}{2}}\left[\Delta A_{i}^{*}(X+\Delta X)^{-1}(A_{i}+\Delta A_{i})+A_{i}^{*}(X+\Delta X)^{-1}\Delta A_{i}\right]X^{-\frac{1}{2}}\right\|
\\
&\leq & \left(\| X^{-\frac{1}{2}}\Delta X
X^{-\frac{1}{2}}\| +
\frac{\|X^{-\frac{1}{2}}\Delta
X X^{-\frac{1}{2}}\|^{2}}{1-\varrho\sum\limits_{i=1}^{m}\|\Delta
A_{i}\|-\omega\|\Delta Q\|}\right)\left(\frac{1}{\beta^{2}}\sum\limits_{i=1}^{m}\|A_{i}\|^{2}\right)\\
&&+\frac{\sum\limits_{i=1}^{m}\|\Delta A_{i}\|(2\| A_{i}\|+\|\Delta
A_{i}\|)}{\beta^{\;2}(1-\varrho\,\sum\limits_{i=1}^{m}\|\Delta A_{i}\|-\omega\|\Delta Q\|)}+\frac{\|\Delta Q\|}{\beta}\\
&\leq & \left(\xi_{1}+\frac{\xi_{1}^{2}}{1-\xi_{1}}\right)\left(\frac{1}{\beta^{2}}
\sum\limits_{i=1}^{m}\|A_{i}\|^{2}\right)+
\frac{s}{\beta^{2}(1-\xi_{1})}+\frac{\|\Delta Q\|}{\beta}\\
&=&\xi_{1}.
\end{eqnarray*}
That is $f(\Omega)\subseteq\Omega.$ By Brouwer fixed point theorem, there exists a $\Delta
X\in\Omega$ such that $f(\Delta X)=\Delta X$. Moreover, by Theorem
\ref{thm5}, we know that $X$ and $\widetilde{X}$ are the unique
solutions to Eq.(\ref{eq1}) and Eq.(\ref{eq7}), respectively. Then
\begin{eqnarray*}
\frac{\|\widetilde{X}-X\|}{\|X\|} & = & \frac{\|\Delta
X\|}{\|X\|}=\frac{\| X^{1/2}(X^{-1/2}\Delta X
X^{-1/2})X^{1/2}\|}{\|X\|}\\
& \leq &\|X^{-1/2}\Delta X X^{-1/2}\|\leq \varrho\sum\limits_{i=1}^{m}\|\Delta
A_{i}\|+\omega\|\Delta Q\|.
\end{eqnarray*}
\end{proof}
\begin{rem}
With $$\varrho\sum\limits_{i=1}^{m}\|\Delta
A_{i}\|+\omega\|\Delta Q\|=\frac{2(\sum\limits_{i=1}^{m}\|\Delta A_{i}\|(2\|A_{i}\|+\|\Delta A_{i}\|)+\beta\|\Delta Q\|)}{b+\sqrt{b^{2}-4\beta^{2}\;(\beta\;\|\Delta Q\|+s)}},$$ we get
$\varrho\sum\limits_{i=1}^{m}\|\Delta
A_{i}\|+\omega\|\Delta Q\|\rightarrow 0$ \;as \;$\Delta Q\rightarrow 0$ and  $\|\Delta A_{i}\|\rightarrow 0\;(i=1,2, \cdots, m).$ Therefore Eq.(\ref{eq1}) is well-posed.
\end{rem}

Next, with the help of the following lemma, we shall derive a new perturbation bound as shown in Theorem \ref{thm8}.

\begin{lemma}\label{lem5}                                                   
Suppose that $X$ is a unique positive definite solution of Eq.(\ref{eq1}). If
\begin{equation}\label{eq12}                                                 
\sum\limits_{i=1}^{m}\|A_{i}\|^{2}<\beta^{2},
\end{equation}
 then $$\|\textmd{d}X\|\leq \frac{2\beta\sum\limits_{i=1}^{m}(\|A_{i}\|\|dA_{i}\|)}{\beta^{2}-\sum\limits_{i=1}^{m}\|A_{i}\|^{2}}.$$
\end{lemma}

\begin{proof}{According to Lemma {\ref{lem4}}, differentiating on both sides of Eq.(\ref{eq1}), we have
$$dX-\sum\limits_{i=1}^{m}[dA_{i}^{*}(X^{-1}A_{i})-(A_{i}^{*}X^{-1})dX(X^{-1}A_{i})+(A_{i}^{*}X^{-1})dA_{i}]=0.$$
Therefore, $$dX+\sum\limits_{i=1}^{m}(A_{i}^{*}X^{-1})dX(X^{-1}A_{i})=\sum\limits_{i=1}^{m}dA_{i}^{*}(X^{-1}A_{i})+
\sum\limits_{i=1}^{m}(A_{i}^{*}X^{-1})dA_{i}$$
and
\begin{eqnarray*}
\|dX+\sum\limits_{i=1}^{m}(A_{i}^{*}X^{-1})dX(X^{-1}A_{i})\|&=& \|\sum\limits_{i=1}^{m}dA_{i}^{*}(X^{-1}A_{i})+
\sum\limits_{i=1}^{m}(A_{i}^{*}X^{-1})dA_{i} \|\\
&\leq &  \sum\limits_{i=1}^{m}\|dA_{i}^{*}\|\|X^{-1}\|\|A_{i}\|+
\sum\limits_{i=1}^{m}\|A_{i}^{*}\|\|X^{-1}\|\|dA_{i}\|\\
&=&  2\sum\limits_{i=1}^{m}\|A_{i}\|\|X^{-1}\|\|dA_{i}\|
\end{eqnarray*}
are true. By Theorem \ref{thm3}, it follows that $\|X^{-1}\|\leq \frac{1}{\beta}.$
Then
\begin{equation}\label{eq14}                                                  
\|dX+\sum\limits_{i=1}^{m}(A_{i}^{*}X^{-1})dX(X^{-1}A_{i})\|\leq \frac{2}{\beta}\sum\limits_{i=1}^{m}\|A_{i}\|\|dA_{i}\|.
\end{equation}
In addition,
\begin{eqnarray}\label{eq15}                                                   
 &&\|d X+\sum\limits_{i=1}^{m}(A_{i}^{*}X^{-1})d X(X^{-1}A_{i})\|
\geq  \|dX\|-\|\sum\limits_{i=1}^{m}(A_{i}^{*}X^{-1})d X(X^{-1}A_{i})\|  \nonumber\\
&\geq& \|dX\|- \sum\limits_{i=1}^{m}\|(A_{i}^{*}X^{-1})d X(X^{-1}A_{i})\|
 \geq \|dX\|-\frac{1}{\beta^{2}}\left(\sum\limits_{i=1}^{m}\|A_{i}\|^{2}\right)\|dX\|\nonumber\\
&=&\left(1-\frac{1}{\beta^{2}}\sum\limits_{i=1}^{m}\|A_{i}\|^{2}\right)\|dX\|.
\end{eqnarray}
By (\ref{eq12}), it follows that $(1-\frac{1}{\beta^{2}}\sum\limits_{i=1}^{m}\|A_{i}\|^{2})\|dX\|>0.$

Combining (\ref{eq14}) and (\ref{eq15}), we obtain
$$(1-\frac{1}{\beta^{2}}\sum\limits_{i=1}^{m}\|A_{i}\|^{2})\|dX\|\leq
\frac{2}{\beta}\sum\limits_{i=1}^{m}\left(\|A_{i}\|\|dA_{i}\|\right),$$ which means that
$$\|dX\|\leq \frac{2\beta\sum\limits_{i=1}^{m}(\|A_{i}\|\|dA_{i}\|)}{\beta^{2}-\sum\limits_{i=1}^{m}\|A_{i}\|^{2}}.$$
}\end{proof}

\begin{theorem}\label{thm8}                                                          
Suppose that $X$, $\widetilde{X}$ are the unique positive definite solutions of Eq.(\ref{eq1}) and Eq. (\ref{eq7}), respectively.
If
\begin{equation}\label{eq17}                                                            
\sum\limits_{i=1}^{m}\|A_{i}\|^{2}<\beta^{2}\;\;\mbox{and}\;\;\sum\limits_{i=1}^{m}(\|A_{i}\|+\|\Delta A_{i}\|)^{2}<\beta^{2},
\end{equation}
then $$\|\widetilde{X}-X\|\leq \frac{2\beta\;\sum\limits_{i=1}^{m}(\|A_{i}\|+\|\Delta A_{i}\|)\|\Delta A_{i}\|}
{\beta^{2}-\sum\limits_{i=1}^{m}(\|A_{i}\|+\|\Delta A_{i}\|)^{2}}$$
and $$\frac{\|\widetilde{X}-X\|}{\|X\|}\leq \frac{2\beta\;\sum\limits_{i=1}^{m}(\|A_{i}\|+\|\Delta A_{i}\|)\|\Delta A_{i}\|}
{(\beta^{2}-\sum\limits_{i=1}^{m}(\|A_{i}\|+\|\Delta A_{i}\|)^{2})\|X\|}\equiv\xi_{2}$$
hold true.
\begin{proof}
{Set $A_{i}(t)=A_{i}+t\Delta A_{i},$ $t\in [0,1].$ By Theorem \ref{thm5}, we have that for arbitrary $t\in [0,1],$
the matrix equation $$X-\sum\limits_{i=1}^{m}A_{i}^{*}(t)X^{-1}A_{i}(t)=Q $$
has a unique positive definite solution $X(t)$ satisfying
$$X(0)=X,\;\;\;X(1)=\widetilde{X}.$$
By Lemma \ref{lem5}  , we have
\begin{eqnarray*}
&&\|\widetilde{X}-X\|=\|X(1)-X(0)\|=\|\int_{0}^{1}d X(t)\|\leq \int_{0}^{1}\|d X(t)\|\\
&\leq& \int_{0}^{1}\frac{2\beta\sum\limits_{i=1}^{m}(\|A_{i}(t)\|\|dA_{i}(t)\|)}
{\beta^{2}-\sum\limits_{i=1}^{m}(\|A_{i}\|+t\|\Delta A_{i}\|)^{2}}
\leq  \int_{0}^{1}\frac{2\beta \sum\limits_{i=1}^{m}(\|A_{i}\|+t\|\Delta A_{i}\|)\|\Delta A_{i}\|}
{\beta^{2}-\sum\limits_{i=1}^{m}(\|A_{i}\|+t\|\Delta A_{i}\|)^{2}}dt.
\end{eqnarray*}
By mean value theorem of integration, there exists $\varepsilon\in[0,1]$ satisfying
\begin{eqnarray*}
\|\widetilde{X}-X\|&\leq& \int_{0}^{1}\frac{2\beta \sum\limits_{i=1}^{m}(\|A_{i}\|+t\|\Delta A_{i}\|)\|\Delta A_{i}\|}
{\beta^{2}-\sum\limits_{i=1}^{m}(\|A_{i}\|+t\|\Delta A_{i}\|)^{2}}dt
=\frac{\sum\limits_{i=1}^{m}2\beta(\|A_{i}\|+\varepsilon\|\Delta A_{i}\|)\|\Delta A_{i}\|}
{\beta^{2}-\sum\limits_{i=1}^{m}(\|A_{i}\|+\varepsilon\|\Delta A_{i}\|)^{2}}\\
&\leq& \frac{\sum\limits_{i=1}^{m}2\beta(\|A_{i}\|+\|\Delta A_{i}\|)\|\Delta A_{i}\|}
{\beta^{2}-\sum\limits_{i=1}^{m}(\|A_{i}\|+\|\Delta A_{i}\|)^{2}}.
 \end{eqnarray*}
}
\end{proof}
\end{theorem}

Next, based on the operator theory, we derive a sharper perturbation estimate.

Subtracting (\ref{eq1}) from (\ref{eq7}) we have
\begin{equation}\label{eq18}                                 
  \Delta X+\sum\limits_{i=1}^{m}B_{i}^{*}\Delta X B_{i}=E+h(\Delta X),
\end{equation}
where
\begin{eqnarray*}
&& B_{i}=X^{-1}A_{i}, \\
&& E=\sum\limits_{i=1}^{m}(B_{i}^{*}\Delta A_{i}+\Delta A_{i}^{*}B_{i})
+\sum\limits_{i=1}^{m}\Delta A_{i}^{*}X^{-1}\Delta A_{i}+\Delta Q, \\
&&h(\Delta X)=\sum\limits_{i=1}^{m}B_{i}^{*}\Delta X X^{-1} \Delta X(I+X^{-1}\Delta X)^{-1}B_{i}  -\sum\limits_{i=1}^{m}\widetilde{A}_{i}^{*}X^{-1}\Delta X(I+X^{-1}\Delta X)^{-1}X^{-1}\Delta A_{i}\\
&& \;\;\;\;\;\;\;\;\;\;-\sum\limits_{i=1}^{m}\Delta A_{i}^{*}X^{-1}\Delta X(I+X^{-1}\Delta X)^{-1} B_{i}.
\end{eqnarray*}

We define the linear operator \textbf{L}: $\mathcal{H}^{n\times n}\rightarrow\mathcal{H}^{n\times n}$ by

\begin{equation*}
\textbf{L}W=W+\sum\limits_{i=1}^{m}B_{i}^{*}WB_{i},\;\;W\in\mathcal{H}^{n\times n}.
\end{equation*}
Since $$X-\sum\limits_{i=1}^{m}B_{i}^{*}XB_{i}=X-\sum\limits_{i=1}^{m}A_{i}^{*}X^{-1}XX^{-1}A_{i}=
X-\sum\limits_{i=1}^{m}A_{i}^{*}X^{-1}A_{i}=Q>0,$$
by Lemma 3.4.1 and Proposition 3.3.1 in \cite{s36}, the operator $\textbf{L}$ is invertible.
We also define operators $ {\textbf P_{i}}:\mathcal{C}^{n\times n}\rightarrow\mathcal{H}^{n\times n}$ by

$${\textbf P_{i}}Z_{\;i}=\textbf{L}^{-1}(B_{i}^{*}Z_{\;i}+Z_{\;i}^{*}B_{i}),\;\;Z_{i}\in\mathcal{C}^{n\times n},\;\;i=1,2, \cdots, m.$$
Thus,we can rewrite (\ref{eq18}) as
\begin{equation}\label{eq19}                                                          
    \Delta X=\textbf{L}^{-1}\Delta Q+\sum\limits_{i=1}^{m}{\textbf P_{i}}\Delta A_{i}+
\textbf{L}^{-1}(\sum\limits_{i=1}^{m}\Delta A_{i}^{*}X^{-1}\Delta A_{i})+\textbf{L}^{-1}(h(\Delta X)).
\end{equation}
Define
$$
||\mathbf{L}^{-1}||=\max_{\begin{array}{c}
W\in\mathcal{H}^{n\times n}\\
||W||=1
\end{array}}||\mathbf{L}^{-1}W||,\;\;\;\;
||\mathbf{P}_{i}||=\max_{\begin{array}{c}
Z\in\mathcal{C}^{n\times n}\\
||Z||=1
\end{array}}||\mathbf{P}_{i}Z||.
$$
 Now we denote
 \begin{eqnarray*}
l&=&\|\textbf{L}^{-1}\|^{-1},\;\;\zeta=\|X^{-1}\|,\;\;m_{i}=\|A_{i}\|,\;\;n_{i}=\|\textbf{P}_{i}\|,\;\;\theta_{i}=\|B_{i}\|,\;\;
\theta=\sum\limits_{i=1}^{m}\theta_{i}^{2},\;i=1,2, \cdots, m,\\
\epsilon&=& \frac{1}{l}\|\Delta Q\|+\sum\limits_{i=1}^{m}(n_{i}\|\Delta A_{i}\|+\frac{\zeta}{l}\|\Delta A_{i}\|^{2}),\;\;\;\;\sigma\;\;=\;\;\frac{\zeta}{l}\sum\limits_{i=1}^{m}((m_{i}+\|\Delta A_{i}\|)\zeta+\theta_{i})\|\Delta A_{i}\|.
 \end{eqnarray*}
  Then we can state the third perturbation estimate as follows.

  \begin{theorem}\label{thm9}                                         
  If
  \begin{equation}\label{eq20}                                             
   \sigma<1\;\;\mbox{and}\;\;\epsilon<\frac{l(1-\sigma)^{2}}{\zeta(l+l\sigma+2\theta+2\sqrt{(l\sigma+\theta)(\theta+l)})},
  \end{equation}
  then
  \begin{equation*}
    \|\widetilde{X}-X\|\leq\frac{2l\epsilon}{l(1+\zeta\epsilon-\sigma)+
    \sqrt{l^{2}(1+\zeta\epsilon-\sigma)^{2}-4l\zeta\epsilon(l+\theta)}}\equiv\xi_{3}.
  \end{equation*}

 \end{theorem}

   \begin{proof}
   {Let $$f(\Delta X)=\textbf{L}^{-1}\Delta Q+\sum\limits_{i=1}^{m}{\textbf P_{i}}\Delta A_{i}+
\textbf{L}^{-1}(\sum\limits_{i=1}^{m}\Delta A_{i}^{*}X^{-1}\Delta A_{i})+\textbf{L}^{-1}(h(\Delta X)).$$ Obviously, $f:\mathcal{H}^{n\times n}\rightarrow\mathcal{H}^{n\times n}$ is continuous. The condition (\ref{eq20}) ensures that the quadratic equation $\zeta(l+\theta)\xi^{2}-l(1+\zeta\epsilon-\sigma)\xi+l\epsilon=0$ with respect to the variable $\xi$ has two positive real roots. The smaller one is
$$\xi_{3}=\frac{2l\epsilon}{l(1+\zeta\epsilon-\sigma)+
    \sqrt{l^{2}(1+\zeta\epsilon-\sigma)^{2}-4l\zeta\epsilon(l+\theta)}}.$$ Define
    $\Omega=\{\Delta X\in \mathcal{H}^{n\times n}: \|\Delta X\|\leq \xi_{3}\}.$ Then for any $\Delta X\in \Omega,$ by (\ref{eq20}), we have
    \begin{eqnarray*}
   && ||X^{-1}\Delta X||\leq ||X^{-1}||||\Delta X||\leq\zeta\;\xi_{3}\leq
\zeta\cdot\frac{2l\epsilon}{l(1+\zeta\epsilon-\sigma)}\\
&&=
1+\frac{\zeta\epsilon+\sigma-1}{1+\zeta\epsilon-\sigma}\leq 1+\frac{-2(1-\sigma)(l\sigma+\theta)}{(l\sigma+l+2\theta)(1+\zeta\epsilon-\sigma)}<1.
\end{eqnarray*}
It follows that $I-X^{-1}\Delta X$ is nonsingular and
$$\|I-X^{-1}\Delta X\|\leq\frac{1}{1-\|X^{-1}\Delta X\|}\leq\frac{1}{1-\zeta\|\Delta X\|}.$$
Therefore, we have
\begin{eqnarray*}
\|f(\Delta X)\|&\leq& \frac{1}{l}\|\Delta Q\|+\sum\limits_{i=1}^{m}(n_{i}\|\Delta A_{i}\|+\frac{\zeta}{l}\|\Delta A_{i}\|^{2})+\frac{1}{l}\sum\limits_{i=1}^{m}\theta_{i}^{2}\frac{\zeta\|\Delta X\|^{2}}{1-\zeta\|\Delta X\|}\\
&+&\frac{\zeta}{l}\sum\limits_{i=1}^{m}\left[\left(\zeta(m_{i}+\|\Delta A_{i}\|)+\theta_{i}\right)\|\Delta A_{i}\|\right]\cdot\frac{\|\Delta X\|}{1-\zeta\|\Delta X\|}\\
&\leq &\epsilon+\frac{\sigma\|\Delta X\|}{1-\zeta\|\Delta X\|}+\frac{\theta\zeta\|\Delta X\|^{2}}{l(1-\zeta\|\Delta X\|)}\\
&\leq &\epsilon+\frac{\sigma\xi_{3}}{1-\zeta\xi_{3}}+\frac{\theta\zeta\xi_{3}^{2}}{l(1-\zeta\xi_{3})}=\xi_{3},
\end{eqnarray*}
for $\Delta X\in \Omega.$ That is $f(\Omega)\subseteq\Omega.$ According to Schauder fixed point theorem, there exists $\Delta X_{*}\in\Omega$ such that $f(\Delta X_{*})=\Delta X_{*}.$ It follows that $X+\Delta X_{*}$ is a Hermitian solution of Eq.(\ref{eq7}). By Theorem \ref{thm5}, we know that the solution of Eq.(\ref{eq7}) is unique. Then $\Delta X_{*}=\widetilde{X}-X$ and $\|\widetilde{X}-X\|\leq\xi_{3}.$
}\end{proof}

\begin{rem}   \label{rem3}                                                                   
From Theorem \ref{thm9}, we get the first order perturbation bound for the solution as follows:
\begin{eqnarray*}
&&\|\widetilde{X}-X\|\leq\frac{1}{l}\|\Delta Q\|+\sum\limits_{i=1}^{m}n_{i}\|\Delta A_{i}\|+O\left(\|(\Delta A_{1}, \Delta A_{2}, \cdots, \Delta A_{m}, \Delta Q)\|_{F}^{2}\right),\\\mbox{as}&&(\Delta A_{1}, \Delta A_{2}, \cdots, \Delta A_{m}, \Delta Q)\rightarrow 0.
\end{eqnarray*}
Combining this with (\ref{eq19}) gives
$$\Delta X=\textbf{L}^{-1}\Delta Q+{\textbf L}^{-1}\sum\limits_{i=1}^{m}(B_{i}^{*}\Delta A_{i}+\Delta A_{i}^{*}B_{i})+O\left(\|(\Delta A_{1}, \Delta A_{2}, \cdots, \Delta A_{m}, \Delta Q)\|_{F}^{2}\right).$$
as $\;\;(\Delta A_{1}, \Delta A_{2}, \cdots, \Delta A_{m}, \Delta Q)\rightarrow 0,$
\end{rem}
\section{Backward error }
In this section, we derive a  backward error of an approximate
solution for the unique solution to Eq. (\ref{eq1}) beginning with the lemma.

\begin{lemma}\label{lem6}                                                                      

For every positive definite matrix $X\in\mathcal{H}^{n\times n}$, if $X+\Delta X\geq (1/\nu)I>0,$ then
$$\|\sum\limits_{i=1}^{m}A_{i}^{*}((X+\Delta X)^{-1}-X^{-1})A_{i}\|
\leq (\|\Delta X \|+\nu \|\Delta X \|^{2})\sum\limits_{i=1}^{m}\|X^{-1}A_{i}\|^{2}.$$
\end{lemma}

\begin{proof}
According to
$$
(X+\Delta X)^{-1}-X^{-1}= -X^{-1}\Delta X(X+\Delta X)^{-1}
 = -X^{-1}\Delta XX^{-1}+X^{-1}\Delta X X^{-1}\Delta X(X+\Delta X)^{-1},
$$
it follows that
\begin{eqnarray*}
 && \|\sum\limits_{i=1}^{m}A_{i}^{*}((X+\Delta X)^{-1}-X^{-1})A_{i}\|  \\
  &\leq& \sum\limits_{i=1}^{m}(\|A_{i}^{*}X^{-1}\Delta XX^{-1}A_{i}\|+\|A_{i}^{*}X^{-1}\Delta X X^{-1}\Delta X(X+\Delta X)^{-1}A_{i}\|)\\
  &\leq&(\|\Delta X\|+\nu\|\Delta X\|^{2}) \sum\limits_{i=1}^{m}\|X^{-1}A_{i}\|^{2}.
 \end{eqnarray*}
\end{proof}

\begin{theorem}\label{thm7}                                                             
Let $\widetilde{X}>0$ be an approximation to the solution $X$ of
Eq.(\ref{eq1}). If
the residual $R(\widetilde{X}) \equiv
Q+\sum\limits_{i=1}^{m}A_{i}^*\widetilde{X}^{-1}A_{i}-\widetilde{X}$ satisfies

\begin{equation}\label{eq11}                                         
\|R(\widetilde{X})\| < \frac{(1-\Sigma)^{2}}{1+\Sigma+2\sqrt{\Sigma}}\;\lambda_{\min}(\widetilde{X}),
    \;\;\mbox{where}\;\;\Sigma\equiv\sum\limits_{i=1}^{m}\|\widetilde{X}^{-1}A_{i}\|^{2}<1,
\end{equation}
then
\begin{equation}\label{eq24}
    \|\widetilde{X}-X\|\leq \theta\|R(\widetilde{X})\|,
\end{equation}
where$$\theta=\frac{2\lambda_{\min}(\widetilde{X})}{(1-\Sigma)\lambda_{\min}(\widetilde{X})
+\|R(\widetilde{X})\|+
\sqrt{((1-\Sigma)\lambda_{\min}(\widetilde{X})
+\|R(\widetilde{X})\|)^{2}-4\lambda_{\min}(\widetilde{X})\|R(\widetilde{X})\|}}.$$
\end{theorem}

\begin{proof}

Let $$\Psi = \{\Delta X \in
\mathcal{H}^{n\times n}:\|\Delta X \|\leq \theta
\|R(\widetilde{X})\|\}.$$ Obviously, $\Psi$ is a nonempty
bounded convex closed set. Let $$g(\Delta
X)=\sum\limits_{i=1}^{m}A_{i}^*\left[(\widetilde{X}+\Delta
X)^{-1}-\widetilde{X}^{-1}\right]A_{i}+R(\widetilde{X}).$$ Evidently $g:
\Psi\mapsto\mathcal{H}^{n\times n}$ is continuous.

Note that the condition (\ref{eq11}) ensures that the quadratical equation
$$x^{2}-\left(\lambda_{\min}(\widetilde{X})(1-\Sigma)+\|R(\widetilde{X})\|\right)x+
\lambda_{\min}(\widetilde{X})\|R(\widetilde{X})\|=0$$ has two positive real roots, and the smaller one
is given by
$$
\mu_{*}=\frac{2\lambda_{\min}(\widetilde{X})\|R
({\widetilde{X})}\|}{(1-\Sigma)\lambda_{\min}(\widetilde{X})
+\|R(\widetilde{X})\|+
\sqrt{((1-\Sigma)\lambda_{\min}(\widetilde{X})
+\|R(\widetilde{X})\|)^{2}-4\lambda_{\min}(\widetilde{X})\|R(\widetilde{X})\|}}.$$
Next, we will prove
that $g(\Psi)\subseteq\Psi$.

For every $\Delta X\in\Psi,$ we have
$$\Delta X \geq-\theta\|R(\widetilde{X})\|I.$$ Hence
 $$\widetilde{X}+\Delta X\geq \widetilde{X}-\theta\|R(\widetilde{X})\|I\geq
 (\lambda_{\min}(\widetilde{X})-\theta\|R(\widetilde{X})\|)I.$$ By (\ref{eq24}), one sees
that
$$\theta\|R(\widetilde{X})\|\leq
\frac{2\lambda_{\min}(\widetilde{X})\|R
({\widetilde{X})}\|}{(1-\Sigma)\lambda_{\min}(\widetilde{X})
+\|R(\widetilde{X})\|}=
\lambda_{\min}(\widetilde{X})\left(1+\frac{\|R(\widetilde{X})\|-(1-\Sigma)\lambda_{\min}(\widetilde{X})
}{(1-\Sigma)\lambda_{\min}(\widetilde{X})+\|R(\widetilde{X})\|}\right).$$
According to (\ref{eq11}), we obtain
\begin{eqnarray*}
\|R(\widetilde{X})\|-(1-\Sigma)\lambda_{\min}(\widetilde{X})&\leq & \left(\frac{(1-\Sigma)^{2}}{1+\Sigma+2\sqrt{\Sigma}}-(1-\Sigma)\right)\lambda_{\min}(\widetilde{X})
 \leq \frac{-2(1-\Sigma)\lambda_{\min}(\widetilde{X})}{1+\Sigma}<0,\\
\end{eqnarray*}
which implies that $$\theta\|R(\widetilde{X})\|\leq \lambda_{\min}(\widetilde{X})\;\;\mbox{and}\;\;(\lambda_{\min}(\widetilde{X})-\theta\|R(\widetilde{X})\|)I>0.$$

According to Lemma \ref{lem6}, we obtain
\begin{eqnarray*}
&&\|g(\Delta X)\|\\
&\leq&\left(\|\Delta X \|+ \frac{\|\Delta X \|^{2}}{\lambda_{\min}(\widetilde{X})-\theta\|R(\widetilde{X})\|}\right)\sum\limits_{i=1}^{m}\|X^{-1}A_{i}\|^{2}+\|R(\widetilde{X})\|\\
&\leq &\left(\theta\|R(\widetilde{X})\|+\frac{(\theta\|R(\widetilde{X})\|)^{2}}{\lambda_{\min}(\widetilde{X})-\theta\|R(\widetilde{X})\|}\right)\Sigma
+\|R(\widetilde{X})\|\\
&=&\theta\|R(\widetilde{X})\|.
\end{eqnarray*}
By Brouwer's fixed point theorem, there exists a $\Delta X\in\Psi$
such that $g(\Delta X)=\Delta X.$ Hence $\widetilde{X}+\Delta X$
is a solution of Eq.(\ref{eq1}). Moreover, by Theorem \ref{thm5}, we
know that the solution $X$ of Eq.(\ref{eq1}) is unique. Then
\begin{equation*}
    \|\widetilde{X}-X\|=\|\Delta X\|\leq\theta\|R(\widetilde{X})\|.
\end{equation*}
\end{proof}
\section{Condition number}
In this section, we apply the theory of condition number developed by Rice \cite{s37} to study condition number of the unique solution to Eq. (\ref{eq1}).
\subsection{The complex case }

Suppose that $X$ and $\widetilde{X}$ are the solutions of Eq.(\ref{eq1}) and Eq.(\ref{eq7}), respectively. Let
$\Delta A=\widetilde{A}-A$, $\Delta Q=\widetilde{Q}-Q$ and $\Delta
X=\widetilde{X}-X$. Using Theorem \ref{thm9} and Remark \ref{rem3}, we have

\begin{equation}\label{eq21}                                                                    
\Delta X=\widetilde{X}-X=\textbf{L}^{-1}\Delta Q+{\textbf L}^{-1}\sum\limits_{i=1}^{m}(B_{i}^{*}\Delta A_{i}+\Delta A_{i}^{*}B_{i})+O\left(\|(\Delta A_{1}, \Delta A_{2}, \cdots, \Delta A_{m}, \Delta Q)\|_{F}^{2}\right),
\end{equation}
as $(\Delta A_{1}, \Delta A_{2}, \cdots, \Delta A_{m}, \Delta Q)\rightarrow 0.$

By the theory of condition number developed by Rice \cite{s37}, we define
the condition number of the Hermitian positive definite solution
$X$ to Eq.(\ref{eq1}) by
\begin{equation}  \label{eq22}                                               
c(X)=\lim_{\delta\rightarrow 0}\sup_{||(\frac{\Delta
A_{1}}{\eta_{1}}, \frac{\Delta
A_{2}}{\eta_{2}}, \cdots, \frac{\Delta
A_{m}}{\eta_{m}},
\frac{\Delta Q}{\rho})||_{F}\leq\delta}\frac{||\Delta
X||_{F}}{\xi\delta},
\end{equation}
where $\xi$,  $\rho$ and $\eta_{i},$ $i=1,2, \cdots, m,$ are positive parameters. Taking
$\xi=\eta_{i}=\rho=1$  in (\ref{eq22}) gives the absolute condition number $c_{abs}(X)$,
and taking $\xi=||X||_{F}$, $\eta_{i}=||A_{i}||_{F}$  and $\rho=||Q||_{F}$ in (\ref{eq22})
gives the relative condition number $c_{rel}(X)$.

Substituting (\ref{eq21}) into (\ref{eq22}), we get
\begin{eqnarray*}  \label{eq23}                                             
c(X)&=&\frac{1}{\xi}\!\!\!\max_{\begin{array}{c}(\frac{\Delta
A_{1}}{\eta_{1}}, \frac{\Delta
A_{2}}{\eta_{2}}, \cdots, \frac{\Delta
A_{m}}{\eta_{m}},
\frac{\Delta Q }{\rho})\neq 0\\
 \Delta A_{i}\in\mathcal{C}^{n\times
n}, \Delta Q\in\mathcal{H}^{n\times n}
\end{array}}
\!\!\!\!\!\!\!\!\frac{||\mathbf{L}^{-1}(\Delta Q+\sum\limits_{i=1}^{m}(B_{i}^{*}\Delta
A_{i}+\Delta A_{i}^{*}B_{i})) ||_{F}}{||(\frac{\Delta
A_{1}}{\eta_{1}}, \frac{\Delta
A_{2}}{\eta_{2}}, \cdots, \frac{\Delta
A_{m}}{\eta_{m}}, \frac{\Delta
Q}{\rho})||_{F}}\nonumber\\
&=&\frac{1}{\xi}\!\!\!\max_{\begin{array}{c}(E_{1}, E_{2}, \cdots, E_{m},
H)\neq 0\\
 E_{i}\in\mathcal{C}^{n\times
n}, H\in\mathcal{H}^{n\times n}
\end{array}}
\!\!\!\!\!\!\!\!\frac{||\mathbf{L}^{-1}(\rho H+\sum\limits_{i=1}^{m}\eta_{i}(B_{i}^{*}E_{i}+E_{i}^{*}B_{i}))
||_{F}}{||(E_{1}, E_{2}, \cdots, E_{m}, H)||_{F}}.
\end{eqnarray*}
Let $L$ be the matrix representation of the linear operator
$\mathbf{L}$. Then it is easy to see that
\begin{equation*}
L=I\otimes I+\sum\limits_{i=1}^{m}B_{i}^{T}\otimes B_{i}^{*}=I\otimes I+\sum\limits_{i=1}^{m}(X^{-1}A_{i})^{T}\otimes (X^{-1}A_{i})^{*}.
\end{equation*}
Let\begin{eqnarray}\label{eq34}
&&L^{-1}=S+\textbf{i}\Sigma,\nonumber\\
&&L^{-1}(I\otimes B_{i}^{*})=L^{-1}(I\otimes
(X^{-1}A_{i})^{*})=U_{i1}+\textbf{i}\Omega_{i1},\nonumber\\
&&L^{-1}(B_{i}^{T}\otimes I)\Pi=L^{-1}((X^{-1}A_{i})^{T}\otimes
I)\Pi=U_{i2}+\textbf{i}\Omega_{i2},\nonumber\\
&& S_{c}=\left[\begin{array}{cc}
S & -\Sigma \\
\Sigma & S
\end{array}\right],\;\;\;\;
U_{i}=\left[\begin{array}{cc}
U_{i1}+U_{i2} & \Omega_{i2}-\Omega_{i1}\\
\Omega_{i1}+\Omega_{i2} & U_{i1}-U_{i2}\end{array}\right],\;\;i=1,2, \cdots, m,
\end{eqnarray}
$$\vector H=x+\textbf{i}y,\;\;\vector E_{i}=a_{i}+\textbf{i}b_{i},
\;\;g=(x^{T}, y^{T}, a^{T}_{1}, b^{T}_{1}, \cdots, a^{T}_{m}, b^{T}_{m})^{T},\;\;M=(E_{1}, E_{2}, \cdots, E_{m},H),$$
where $x, y, a_{i}, b_{i}
\in\mathcal{R}^{n^{2}},\; S, \Sigma, U_{i1}, U_{i2}, \Omega_{i1}, \Omega_{i2}\in\mathcal{R}^{n^{2}\times n^{2}},\;i=1, 2, \cdots, m,$ \;$\Pi$ is the vec-permutation matrix, such that
$$\vector \;E_{i}^{T}=\Pi \;\vector \;E_{i}.$$
Furthermore, we obtain that
\begin{eqnarray*}
&&c(X)=\frac{1}{\xi}\max_{\begin{array}{c}M\neq 0\\
\end{array}}
\frac{||\mathbf{L}^{-1}(\rho H+\sum\limits_{i=1}^{m}\eta_{i}(B_{i}^{*}E_{i}+E_{i}^{*}B_{i}))
||_{F}}{||(E_{1}, E_{2}, \cdots, E_{m}, H)||_{F}}\\
&=&\frac{1}{\xi}\max_{\begin{array}{c}M\neq 0\\
\end{array}}
\frac{||\rho{L}^{-1} \vector H+\sum\limits_{i=1}^{m}\eta_{i}{L}^{-1}((I\otimes B_{i}^{*})\vector E_{i}+(B_{i}^{T}\otimes I)\vector E_{i}^{*})
||}{\left\|\left(
\vector E_{1},
\vector E_{2},
\cdots,
\vector E_{m}, \vector H
\right)\right\|}\\
&=&\frac{1}{\xi}\!\!\max_{\begin{array}{c}M\neq 0\\
\end{array}}
\!\!\!\!\!\frac{||\rho(S+\textbf{i}\Sigma)(x+\textbf{i}y)+
\sum\limits_{i=1}^{m}\eta_{i}[(U_{i1}+\textbf{i}\Omega_{i1})(a_{i}+\textbf{i}b_{i})
+(U_{i2}+\textbf{i}\Omega_{i2})(a_{i}-\textbf{i}b_{i})]
||}{\left\|\left(
\vector E_{1},
\vector E_{2},
\cdots,
\vector E_{m}, \vector H
\right)\right\|}\\
&=&\frac{1}{\xi}\!\!\!\max_{\begin{array}{c}g\neq 0\\
\end{array}}
\frac{||(\rho\; S_{c}, \eta_{1}U_{1}, \eta_{2}U_{2}, \cdots, \eta_{m}U_{m})g
||}{\|g\|}\\
&=&\frac{1}{\xi}\;||\;(\rho S_{c},\;\eta_{1}U_{1},\;\eta_{2}U_{2}, \cdots, \eta_{m} U_{m})||,\;\;E_{i}\in\mathcal{C}^{n\times
n}, H\in\mathcal{H}^{n\times n}.
\end{eqnarray*}
Then we have the following theorem.
\begin{theorem}\label{thm10}                                           
The condition number $c(X)$ defined by (\ref{eq22}) has the
explicit expression
\begin{equation}\label{eq31}                                        
c(X)=\frac{1}{\xi}\;||\;(\rho S_{c},\;\eta_{1}U_{1},\;\eta_{2}U_{2}, \cdots, \eta_{m} U_{m})||,
\end{equation}
where the matrices $S_{c}$ and $U_{i}$ are defined as in (\ref{eq34}).
\end{theorem}
\begin{rem}\label{rem4}                                          
From (\ref{eq31}) we have the relative condition number
\begin{equation*} \label{eq32}                                               
c_{rel}(X)=\frac{||\;(||Q||_{F}S_{c},\;||A_{1}||_{F}
U_{1},\;||A_{2}||_{F}
U_{2},\;\cdots, ||A_{m}||_{F}
U_{m})||}{||X||_{F}}.
\end{equation*}
\end{rem}
\subsection{The real case  }
In this subsection we consider the real case, i.e.,
all the coefficient matrices $A_{i}$, $Q$ of Eq.(\ref{eq1}) are real. In
such a case the corresponding solution $X$ is also real.
Completely similar arguments as Theorem \ref{thm10} give the following
theorem.
\begin{theorem}
Let $A_{i}$, $Q$ be real and $c(X)$ be the condition number defined by
(\ref{eq22}). Then $c(X)$ has the explicit expression
\begin{equation*}
c(X)=\frac{1}{\xi}\;||\;(\rho S_{r},\;\eta_{1}U_{1},\;\eta_{2}U_{2}, \cdots, \eta_{m}U_{m})\;||,
\end{equation*}
where
\begin{eqnarray*}
&&S_{r}=\left(I+\sum\limits_{i=1}^{m}(A_{i}^{T}X^{-1})\otimes(A_{i}^{T}X^{-1})\right)^{-1},\\
&&U_{i}=S_{r}[I\otimes(A_{i}^{T}X^{-1})+((A_{i}^{T}X^{-1})\otimes I)\Pi],\;\;i=1, 2, \cdots, m.
\end{eqnarray*}
\end{theorem}

\begin{rem}\label{rem2}                                              
In the real case the relative condition number is given by
\begin{equation*}
c_{rel}(X)=\frac{||\;(||Q||_{F}S_{r},\;||A_{1}||_{F}U_{1},\;||A_{2}||_{F}U_{2}, \cdots,||A_{m}||_{F}
U_{m})||}{||X||_{F}}.
\end{equation*}
\end{rem}
\section{Numerical Examples }
To illustrate the theoretical results of the previous sections, in this
section four simple examples are given, which were carried out
using MATLAB 7.1. For the stopping criterion we take
$\varepsilon_{k+1}(X)=\|X-\sum\limits_{i=1}^{m}A_{i}^{*}X^{-1}A_{i}-I\|<1.0e-10.$

\begin{example}      \label{ex1}                                                                    
We study the matrix equation $$X-A_{1}^{*}X^{-1}A_{1}-A_{2}^{*}X^{-1}A_{2}=I,$$
with \[A_{k}=\frac{\frac{1}{k+2}+2\times 10^{-2}}{||A||}A,\;\;k=1,2,\;\;\;A=\left(\begin{array}{ccccc}
2 & 1& 0& 0& 0\\
1 & 2& 1& 0& 0\\
0 & 1& 2& 1& 0\\
0 & 0& 1& 2& 1\\
0 & 0& 0& 1& 2
\end{array}\right).
\] By computation, $\beta=1.0009,$ $\alpha=1.1976.$
Let $X_{0}=1.1 I.$
Algorithm (\ref{eq6}) needs 11 iterations to obtain the unique positive definite solution
\[X=\left(\begin{array}{ccccc}
 1.0643  &  0.0494  &  0.0104  & -0.0009 &  -0.0000\\
    0.0494 &   1.0747   & 0.0485 &   0.0104  & -0.0009\\
    0.0104  &  0.0485&    1.0747  &  0.0485   & 0.0104\\
   -0.0009   & 0.0104 &   0.0485  &  1.0747 &   0.0494\\
   -0.0000  & -0.0009   & 0.0104  &  0.0494  &  1.0643
\end{array}\right)\in [\beta I, \alpha I]\] with the residual
$\|X-A_{1}^{*}X^{-1}A_{1}-A_{2}^{*}X^{-1}A_{2}-I\|=4.8477e-011,$ which satisfies Theorem \ref{thm5} and Theorem \ref{thm3}.
\end{example}


\begin{example}                     \label{ex2}                                      
We consider the matrix equation  $$X-A_{1}^{*}X^{-1}A_{1}-A_{2}^{*}X^{-1}A_{2}=I,$$
with \[A_{1}=\frac{\frac{1}{3}+2\times 10^{-2}}{||A||}A, \;\;A_{2}=\frac{\frac{1}{6}+3\times 10^{-2}}{||A||}A,\;\;\;A=\left(\begin{array}{ccccc}
2 & 1& 0& 0& 0\\
1 & 2& 1& 0& 0\\
0 & 1& 2& 1& 0\\
0 & 0& 1& 2& 1\\
0 & 0& 0& 1& 2
\end{array}\right).
\]
Suppose that the
coefficient matrices $A_{1}$ and $A_{2}$ are perturbed to
$\widetilde{A_{i}}=A_{i}+\Delta A_{i},i=1,2$, where $$\Delta A_{1}=\frac{10^{-j}}{\|C^{T}+C\|}(C^{T}+C),\;\;\Delta A_{2}=\frac{3\times10^{-j-1}}{\|C^{T}+C\|}(C^{T}+C)$$ and $C$ is a random matrix generated by
MATLAB function \textbf{randn}.

We now consider the corresponding  perturbation bounds for the solution $X$ in Theorem \ref{thm6}, Theorem \ref{thm8}
and Theorem \ref{thm9}.

The conditions in Theorem \ref{thm6} are
\begin{eqnarray*}
con1&=& 2\beta^{2}-b>0,\;\;con2=\beta^{2}-\sum\limits_{i=1}^{2}\|A_{i}\|^{2}>0,\\
  con3 &=& (\beta^{2}-\sum\limits_{i=1}^{2}\|A_{i}\|^{2})^{2}-4\beta^{2}\sum\limits_{i=1}^{2}\|\Delta A_{i}\|\;(2\|A_{i}\|+\|\Delta A_{i}\|)\geq 0.
\end{eqnarray*}

The condition in Theorem \ref{thm8} is
$$con4=\beta^{2}-\sum\limits_{i=1}^{2}(\|A_{i}\|+\|\Delta A_{i}\|)^{2}>0.$$

The conditions in Theorem \ref{thm9} are
\begin{eqnarray*}
con5&=& 1-\sigma>0,\;\;
  con6 = \frac{l(1-\sigma)^{2}}{\zeta(l+l\sigma+2\theta+2\sqrt{(l\sigma+\theta)(\theta+l)})}-\epsilon>0.
\end{eqnarray*}
By computation, we list them in Table \ref{tab:1}.
\begin{table}[h t b]
\caption{Conditions for Example \ref{ex2} with different values of j}
\label{tab:1}
\begin{tabular}{p{1.5cm}p{2.4cm}p{2.4cm}p{2.4cm}p{2.4cm}}
\hline\noalign{\smallskip}
$j$                & 4             & 5              & 6              & 7\\
\noalign{\smallskip}\hline\noalign{\smallskip}
$con1$ & $1.1650$ & $1.1650$ & $1.1650$ & $1.1650$  \\
$con2$ & $0.8379$ & $0.8379$ & $ 0.8379$ & $0.8379$\\
$con3$ & $0.7018$ & $0.7021$ & $ 0.7021$ & $0.7021$\\
$con4$ & $0.8378$ & $0.8379$ & $ 0.8379$ & $0.8379$\\
$con5$ & $0.9999$ & $1.0000$ & $1.0000$ & $1.0000$\\
$con6$ & $0.4802$ & $0.4804$ & $ 0.4804$ & $0.4804$\\ \hline
\noalign{\smallskip}
\end{tabular}
\end{table}

The results listed in Table \ref{tab:1} show that the conditions of Theorem \ref{thm6}- \ref{thm9} are satisfied.
\vskip 0.1in
By Theorem \ref{thm6}-\ref{thm9}, we can
compute the relative perturbation bounds $\xi_{1}, \xi_{2}, \nu_{*}\equiv\frac{\xi_{3}}{\|X\|},$  respectively. These results averaged as the geometric mean of  20 randomly perturbed runs.
Some results are listed in Table \ref{tab:2}.
\begin{table}[h t b]
\caption{Results for Example \ref{ex2} with different values of j}
\label{tab:2}
\begin{tabular}{p{1.5cm}p{2.4cm}p{2.4cm}p{2.4cm}p{2.4cm}}
\hline\noalign{\smallskip}
$j$                & 4             & 5              & 6              & 7\\
\noalign{\smallskip}\hline\noalign{\smallskip}
$\!\!\frac{\|\widetilde{X}-X\|}{\|X\|}$ & $2.7093\times 10^{-5}$ & $2.5933\times 10^{-6}$ & $2.5409\times10^{-7}$ & $2.5031\times10^{-8}$  \\
 $\xi_{1}$ &  $
9.9282\times 10^{-5}$ & $9.9853\times 10^{-6}$ & $ 9.7137\times
10^{-7}$ & $9.8301\times10^{-8}$\\
 $\xi_{2}$ &  $
8.6930\times 10^{-5}$ & $8.7421\times 10^{-6}$ & $ 8.5042\times
10^{-7}$ & $8.6061\times10^{-8}$\\
   $\nu_{*}$ &  $
6.4687\times 10^{-5}$ & $6.5057\times 10^{-6}$ & $ 6.3287\times
10^{-7}$ & $6.4045\times10^{-8}$\\ \hline
\noalign{\smallskip}
\end{tabular}
\end{table}

The results listed in Table \ref{tab:2} show that the perturbation
bound $\nu_{*}$ given by Theorem \ref{thm9} is fairly sharp,  the bound $\xi_{2}$ given by Theorem \ref{thm8}
is relatively sharp, while the bound $\xi_{1}$ given by Theorem \ref{thm6} which does not depend on the exact solution is conservative.
\end{example}

\begin{example}\label{ex3}
We consider  $$X-A_{1}^{*}X^{-1}A_{1}-A_{2}^{*}X^{-1}A_{2}=Q,$$
with \[A_{1}=\frac{\frac{1}{3}+2\times 10^{-2}}{||A||}A, \;\;A_{2}=\frac{\frac{1}{6}+3\times 10^{-2}}{||A||}A,\;\;\;Q=A=\left(\begin{array}{ccccc}
2 & 1& 0& 0& 0\\
1 & 2& 1& 0& 0\\
0 & 1& 2& 1& 0\\
0 & 0& 1& 2& 1\\
0 & 0& 0& 1& 2
\end{array}\right).
\]                                                          
Choose $\widetilde{X}_0=A$. Let the
approximate solution $\widetilde{X}_k$ of $X$ be given with the
iterative method (\ref{eq6}), where $k$ is the iterative number.

The residual $R(\widetilde{X}_k)\equiv
Q+A_{1}^*\widetilde{X}_k^{-1}A_{1}+A_{2}^*\widetilde{X}_k^{-1}A_{2}-\widetilde{X}_k$ satisfies the conditions in
Theorem \ref{thm7}.  By Theorem \ref{thm7}, we can compute the backward error
bound for $\widetilde{X}_k$
$$\parallel\widetilde{X}_k -X\parallel \leq
\theta\|R(\widetilde{X}_k)\|,$$ where $$\theta=\frac{2\lambda_{\min}(\widetilde{X}_{k})}{(1-\Sigma)\lambda_{\min}(\widetilde{X}_{k})
+\|R(\widetilde{X}_{k})\|+
\sqrt{((1-\Sigma)\lambda_{\min}(\widetilde{X}_{k})
+\|R(\widetilde{X}_{k})\|)^{2}-4\lambda_{\min}(\widetilde{X}_{k})\|R(\widetilde{X}_{k})\|}}.$$
 Some results are
listed in Table\ref{tab:3}.
\begin{table}[h t b]
\caption{Results for Example \ref{ex3} with different values of k}
\label{tab:3}
\begin{tabular}{p{1.5cm}p{2.4cm}p{2.4cm}p{2.4cm}p{2.4cm}}
\hline\noalign{\smallskip}
$k$           & 1     &2            & 3               & 4              \\
\noalign{\smallskip}\hline\noalign{\smallskip}
  $||\widetilde{X}_k-X||$&$5.0268\times 10^{-4}$ & $5.7662\times 10^{-6}$ &$6.6162\times 10^{-8}$ & $7.5024\times 10^{-10}$ \\ \hline
 $\theta|| R(\widetilde{X}_k)||$ & $5.1435\times 10^{-4}$ &$ 5.9000\times 10^{-6}    $&$6.7689\times 10^{-8}$ & $ 7.7656\times 10^{-10}$  \\ \hline
\noalign{\smallskip}
\end{tabular}
\end{table}

The results listed in Table \ref{tab:3} show that the error bound
given by Theorem \ref{thm7} is fairly sharp.
\end{example}

\begin{example}           \label{ex4}                                                   
We study the matrix equation $$X-A_{1}^{*}X^{-1}A_{1}-A_{2}^{*}X^{-1}A_{2}=Q,$$
with \[A_{j}=\frac{\frac{1}{j+2}+2\times 10^{-k}}{||A||}A,\;j=1,2,\;\;A=\left(\begin{array}{ccccc}
2 & 1& 0& 0& 0\\
1 & 2& 1& 0& 0\\
0 & 1& 2& 1& 0\\
0 & 0& 1& 2& 1\\
0 & 0& 0& 1& 2
\end{array}\right),\;\;
Q=\left(\begin{array}{ccccc}
2 & 1& 0& 9& 0\\
1 & 2& 1& 0& 8\\
5 & 1& 2& 1& 6\\
9 & 0& 1& 2& 1\\
0 & 2& 3& 1& 2
\end{array}\right).
\]  By Remark \ref{rem2}, we can compute the relative
condition number $c_{rel}(X).$ Some results are listed in Table
\ref{tab:4}.
\begin{table}[h t b]
\caption{Results for Example \ref{ex4} with different values of $k$}
\label{tab:4}
\begin{tabular}{p{2cm}p{1.7cm}p{1.7cm}p{1.7cm}p{1.7cm}p{1.7cm}}
\hline\noalign{\smallskip}
  $k$ & 1 & 3 & 5 & 7 & 9 \\ \noalign{\smallskip}\hline\noalign{\smallskip}
 $c_{rel}(X)$& 1.2704 &  1.0951 & 1.0939 & 1.0938 & 1.0938  \\ \hline\noalign{\smallskip}
\end{tabular}
\end{table}

The numerical results listed in the second line show that the unique
positive definite solution $X$ is well-conditioned.
\end{example}
\section*{Acknowledgements}
The author wishes to express her gratitude to the referees for their fruitful comments
and suggestions regarding the earlier version of this paper.

\end{document}